\documentclass[11pt,reqno]{amsart}
\usepackage{amsmath,amssymb,bbm,url}

\usepackage{tikz}
\usetikzlibrary{calc}
\usetikzlibrary{plotmarks}
\newtheorem{theorem}{Theorem}[section]
\newtheorem{lemma}[theorem]{Lemma}

\newtheorem{corollary}[theorem]{Corollary}
\newtheorem{conjecture}[theorem]{Conjecture}
\newtheorem*{remarks}{Remarks}

\numberwithin{equation}{section}

\newcommand{\e}{\varepsilon}
\newcommand{\R}{{\mathbb R}}
\newcommand{\Real}{\operatorname{Re}}
\newcommand{\tr}{\operatorname{tr}}
\newcommand{\Z}{{\mathbb Z}}

\begin{document}

\title[]{Sums of Laplace eigenvalues --- rotationally symmetric maximizers in the plane}

\author[]{R. S. Laugesen and B. A. Siudeja}
\address{Department of Mathematics, University of Illinois, Urbana,
IL 61801, U.S.A.} \email{Laugesen\@@illinois.edu,Siudeja\@@illinois.edu}
\date{\today}


\begin{abstract}
The sum of the first $n \geq 1$ eigenvalues of the Laplacian is shown to be maximal among triangles for the equilateral triangle, maximal among parallelograms for the square, and maximal among ellipses for the disk, provided the ratio $\text{(area)}^3/\text{(moment of inertia)}$ for the domain is fixed. This result holds for both Dirichlet and Neumann eigenvalues, and similar conclusions are derived for Robin boundary conditions and Schr\"{o}dinger eigenvalues of potentials that grow at infinity. A key ingredient in the method is the tight frame property of the roots of unity.

For general convex plane domains, the disk is conjectured to maximize sums of Neumann eigenvalues.
\end{abstract}

\maketitle

\vspace*{-12pt}

\section{\bf Introduction}

Eigenvalues of the Laplacian represent frequencies in wave motion, rates of decay in diffusion, and energy levels in quantum mechanics. Eigenvalues are challenging to understand: they are known in closed form on only a handful of domains. This difficulty has motivated considerable work on estimating eigenvalues in terms of simpler, geometric quantities such as area and perimeter.

We will obtain a sharp bound on the sum of the first $n \geq 1$ eigenvalues of linear images of rotationally symmetric domains. Our methods apply equally well to Dirichlet, Robin, and Neumann boundary conditions.

Write $\lambda_1, \lambda_2, \lambda_3, \ldots$ for the Dirichlet eigenvalues of the Laplacian on a plane domain of area $A$ and moment of inertia $I$. We will prove that for each $n \geq 1$, the normalized, scale invariant eigenvalue sum
\begin{equation} \label{introeq1}
(\lambda_1 + \dots + \lambda_n) \frac{A^3}{I}
\end{equation}
is maximal among triangular domains for the equilateral triangle. Among parallelograms the maximizer is the square, and the disk is the maximizer among ellipses. The only case known previously was the fundamental tone, $n=1$, due to P\'{o}lya \cite{P52}.

An analogous result will be shown for the sum of Neumann eigenvalues,
\begin{equation} \label{introeq2}
(\mu_1 + \dots + \mu_n) \frac{A^3}{I} ,
\end{equation}
and then for Robin and Sch\"{o}dinger eigenvalues too. These latter results are new even for $n=1$. See Section~\ref{results}.

Our work suggests conjectures for general convex domains. Is the Dirichlet eigenvalue sum \eqref{introeq1} maximal for the disk? Not when $n=1$, curiously, because any rectangle or equilateral triangle gives a larger value for the fundamental tone. We conjecture that those domains maximize $\lambda_1 A^3/I$. For the Neumann eigenvalue sum \eqref{introeq2} it does seem plausible to conjecture maximality for the disk, as we discuss in Section~\ref{open}.

Central to the paper is a new technique we call the
\[
\text{``Method of Rotations and Tight Frames''.}
\]
The idea is to linearly transplant the eigenfunctions of the extremal domain and then average with respect to allowable rotations of that domain. This averaging of the Rayleigh quotient is accomplished using a ``tight frame'' or Parseval identity for the root-of-unity vectors. The Hilbert--Schmidt norm of the linear transformation arises naturally in such averaging, and is then represented in terms of the moment of inertia.

\subsection*{Intuition}
The eigenvalue sum \eqref{introeq1} can be written as a product of two scale invariant factors, as
\[
(\lambda_1 + \dots + \lambda_n) A \cdot \frac{A^2}{I} .
\]
The first factor, $(\lambda_1 + \dots + \lambda_n) A$, normalizes by the area of the domain, and so can be thought of as a generalized ``Faber--Krahn'' term (although the Faber--Krahn theorem says $\lambda_1 A$ is minimal for the disk, not maximal).

The second factor, $A^2/I$, is purely geometric and measures the ``deviation from roundness'' of the domain. This factor is small when the domain is elongated, and hence it balances the largeness of the first factor on such domains.

%Physically, our results say that the frequencies of vibration of a membrane are controlled by a combination of its moment of inertia and area. 
The motivating intuition is that for a domain with characteristic length scales $a$ and $b$, we have
\[
\lambda_1 \
\lesssim \ \frac{1}{a^2} + \frac{1}{b^2} = \frac{ab(a^2+b^2)}{(ab)^3} \ \lesssim \ \frac{I}{A^3} .
\]
This rough calculation is exact for rectangles, up to constant factors.

The first inequality in the calculation extends readily to higher dimensions, but the algebraic identity in the middle becomes more complicated. Thus in higher dimensions it seems that the moment of inertia should be evaluated instead on some kind of ``reciprocal'' domain that has length scales $1/a$ and $1/b$ and so on. Higher dimensional results of this nature will be developed in a later paper \cite{LS11c}; the maximizing domains are regular tetrahedra, cubes, and other Platonic solids.

\subsection*{Related work}

The major contribution of this paper is that it proves upper bounds that are geometrically sharp, on eigenvalue sums of arbitrary length. We do not know any similar results in the literature.

Some results of a different type are known, as we now describe. A bound due to Kr\"{o}ger \cite{K92} for Neumann eigenvalues says that $(\mu_1 + \dots + \mu_n)A/n^2 \leq 2\pi$. The inequality is asymptotically sharp for each domain, because $\mu_n A \sim 4\pi n$ by the Weyl asymptotics. But Kr\"{o}ger's bound is not geometrically sharp for fixed $n$, because there are no domains for which equality holds. Kr\"{o}ger's result should be viewed as a weak version of the P\'{o}lya conjecture. That conjecture asserts that the Weyl asymptotic estimate is in fact a strict upper bound on each Neumann eigenvalue. It has been proved for tiling domains by Kellner \cite{K66}, and up through the third eigenvalue for simply connected plane domains by Girouard, Nadirashvili and Polterovich \cite{GNP09}, but it remains open in general.

Kr\"{o}ger also proved an upper bound on Dirichlet eigenvalue sums, involving $\e$-neighborhoods of the boundary \cite{K94}. This bound is again not geometrically sharp. Kr\"{o}ger's estimates were generalized to domains in homogeneous spaces by Strichartz \cite{S96}.

Weak versions of P\'{o}lya's conjectured lower bound for Dirichlet eigenvalues \cite{P61} are due to Berezin \cite{B72} and Li and Yau \cite{LY83}, with later developments by Laptev \cite{L97} and others using Riesz means and ``universal'' inequalities, as surveyed by Ashbaugh \cite{A02}. Useful \emph{upper} bounds on eigenvalue sums in terms of other eigenvalue sums have lately been obtained this way, by Harrell and Hermi \cite[Corollary 3.1]{HH08}. Note P\'{o}lya's lower bound has been investigated also for eigenvalues under a constant magnetic field, by Frank, Loss and Weidl \cite{FLW09}.

There is considerable literature on low eigenvalues of domains constrained by perimeter, in-radius, or conformal mapping radius, rather than moment of inertia. We summarize this literature in Section~\ref{literature}.

Eigenvalues of triangular domains have been studied a lot, in recent years \cite{AF06,AF08,F06,F07,FS10,LS09,LS10,LR08,S07,S10}, and this paper extends the theory of their upper bounds. Lower bounds on Dirichlet eigenvalues of triangles are proved in a companion paper \cite{LS11a}: there the triangles are normalized by diameter (rather than area and moment of inertia) and equilateral triangles are shown to minimize (rather than maximize) the eigenvalue sums.

For broad surveys of isoperimetric eigenvalue inequalities, one can consult the
monographs of Bandle \cite{B80}, Henrot \cite{He06}, Kesavan
\cite{K06} and P\'{o}lya--Szeg\H{o} \cite{PS51}, and the survey paper by Ashbaugh \cite{A99}.

\section{\bf Assumptions and notation}
\label{notation}

\subsection*{Eigenvalues}
For a bounded plane domain $D$, we denote the Dirichlet eigenvalues of the Laplacian by $\lambda_j(D)$, the Robin eigenvalues by $\rho_j(D;\sigma)$ where the constant $\sigma>0$ is the Robin parameter, and the Neumann eigenvalues by $\mu_j(D)$. In the Robin and Neumann cases we make the standing assumption that the domain has Lipschitz boundary, so that the spectra are well defined. Denoting the eigenfunctions by $u_j$ in each case, we have
\[
\begin{cases}
-\Delta u_j = \lambda_j u_j \quad \text{in $D$} \\
\hfill u_j = 0 \quad \text{on $\partial D$}
\end{cases}
\quad
\begin{cases}
\hfill -\Delta u_j = \rho_j u_j \quad \text{in $D$} \\
\frac{\partial u_j}{\partial n}  + \sigma u_j = 0 \quad \text{on $\partial D$}
\end{cases}
\quad
\begin{cases}
-\Delta u_j = \mu_j u_j \quad \text{in $D$} \\
\hfill \frac{\partial u_j}{\partial n}  = 0 \quad \text{on $\partial D$}
\end{cases}
\]
and
\[
0 < \lambda_1 < \lambda_2 \leq \lambda_3 \leq \dots \qquad
0 < \rho_1 < \rho_2 \leq \rho_3 \leq \dots \qquad
0 = \mu_1 < \mu_2 \leq \mu_3 \leq \dots .
\]
The Robin case reduces to Neumann when $\sigma=0$, and formally reduces to the Dirichlet case when $\sigma=\infty$.

The corresponding Rayleigh quotients are
\begin{align*}
\text{Dirichlet:} \qquad R[u] &  = \frac{\int_D |\nabla u|^2 \, dx}{\int_D u^2 \, dx} && \text{for\ } u \in H^1_0(D) , \\
\text{Robin:} \qquad R[u] &  = \frac{\int_D |\nabla u|^2 \, dx + \sigma \int_{\partial D} u^2 \, ds}{\int_D u^2 \, dx} && \text{for\ } u \in H^1(D) , \\
\text{Neumann:} \qquad R[u] &  = \frac{\int_D |\nabla u|^2 \, dx}{\int_D u^2 \, dx} && \text{for\ } u \in H^1(D) .
\end{align*}
The \textbf{Rayleigh--Poincar\'{e} Principle} \cite[p.~98]{B80} characterizes the sum of the first $n \geq 1$ eigenvalues as
\begin{align*}
& \lambda_1 + \dots + \lambda_n \\
& = \min \big\{ R[v_1] + \dots + R[v_n] : v_1, \dots,v_n \in H^1_0(D)\text{\ are pairwise orthogonal in $L^2(D)$} \big\}
\end{align*}
in the Dirichlet case, and similarly in the Robin and Neumann cases (using trial functions in $H^1$ instead of $H^1_0$).

\subsection*{Geometric quantities}
Let
\begin{itemize}
  \item[] $A=$ area,
  \item[] $I=$ moment of inertia about the centroid.
\end{itemize}
That is,
\[
  I(D) = \int_D |x-\overline{x}|^2 \, dx
\]
where the centroid is $\overline{x} = \frac{1}{A(D)} \int_D x \, dx$.

Given a matrix $M$, write its Hilbert--Schmidt norm as
\[
\lVert M \rVert_{HS}=\big( \sum_{j,k} M_{jk}^2 \big)^{\! 1/2} = (\tr M M^\dagger)^{\! 1/2}
\]
where $M^\dagger$ denotes the transposed matrix.

\section{\bf Sharp upper bounds on eigenvalue sums} \label{results}

\subsection*{Dirichlet and Neumann eigenvalues} Our first result examines the effect on eigenvalues of linearly transforming a rotationally symmetric domain, like in Figure~\ref{fig:lineartrans}.
\begin{figure}[ht]
  \begin{center}
    \begin{tikzpicture}[scale=2]
      \draw +(90:1) -- +(210:1) -- %node [below,blue] {$U$ rotation}
      +(330:1) -- cycle;
      \fill circle (0.02);
      \draw[blue,->] (0:0.2) arc (0:270:0.2);
      \draw[cm={1,1,0,1,(4,0)}] [rotate=60] +(110:1) -- +(230:1) -- +(350:1) -- cycle;
      \fill [cm={1,1,0,1,(4,0)}] circle (0.02);
      \draw[blue,->] (1,0.5) [in=160,out=20] to (3,0.5) node [pos=0.5,above] {$T$}
      node [pos=0.5,below] {linear};
    \end{tikzpicture}
  \end{center}
  \caption{A domain $D$ with rotational symmetry, and its image under a linear map $T$.}
  \label{fig:lineartrans}
\end{figure}
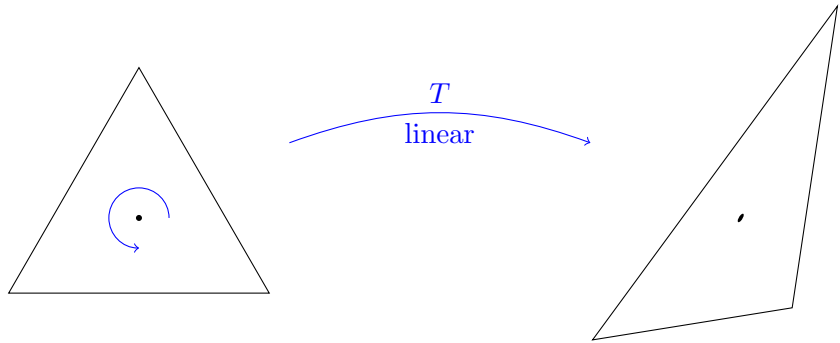
\begin{theorem} \label{2dimDN}
If $D$ has rotational symmetry of order greater than or equal to $3$, then
\begin{equation} \label{DNeq}
(\lambda_1 + \dots + \lambda_n) \big|_{T(D)} \leq
\frac{1}{2} \lVert T^{-1} \rVert_{HS}^2 \, (\lambda_1 + \dots + \lambda_n) \big|_D
\end{equation}
for each $n \geq 1$ and each invertible linear transformation $T$ of the plane. The same inequality holds for the Neumann eigenvalues.

Equality holds in \eqref{DNeq} for the first Dirichlet eigenvalue ($n=1$) if and only if either
\begin{enumerate}
\item[(i)] $T$ is a scalar multiple of an orthogonal matrix, or
\item[(ii)] $D$ is a square and $T(D)$ is a rectangle (possibly with sets of capacity zero removed).
\end{enumerate}
Equality holds for the second Neumann eigenvalue ($n=2$) if and only if
$T$ is a scalar multiple of an orthogonal matrix.
\end{theorem}
The proof is in Section~\ref{2dimDN_proof}. Notice equality does hold in the theorem when $T$ is a scalar multiple of an orthogonal matrix, because if $T=rS$ where $S$ is orthogonal, then $\lambda_j \big( T(D) \big) = r^{-2} \lambda_j(D)$ by rescaling and rotation, and $\frac{1}{2} \lVert T^{-1} \rVert_{HS}^2 = r^{-2}$.

The rotationally symmetric domain $D$ in the theorem need not be convex, need not be a regular polygon, and need not have any axis of symmetry. For example, it could be shaped like a three-bladed propeller.

P\'{o}lya obtained the theorem for $n=1$ (the Dirichlet fundamental tone), although with no equality statement. He stated this result in \cite{P52}, and P\'{o}lya and Schiffer proved it along with results for torsional rigidity and capacity in \cite[Chapter~IV]{PS53}. Our method differs subtly from theirs, as we explain in Section~\ref{2dimDN_proof}, and this difference allows us to handle higher eigenvalue sums and Neumann eigenvalues too.

To express the theorem more geometrically, we observe that the Hilbert--Schmidt norm of the transformation $T^{-1}$ can be expressed in terms of moment of inertia and area (Lemma~\ref{2dimHSnorm}). Hence in particular:
\begin{corollary} \label{regularpolyDN}
Among triangles, the normalized Dirichlet eigenvalue sum
\begin{equation} \label{DNquantity}
(\lambda_1 + \dots + \lambda_n) \frac{A^3}{I}
\end{equation}
is maximal for the equilateral triangle, for each $n \geq 1$. When $n=1$, every maximizer is equilateral.

Among parallelograms, the quantity \eqref{DNquantity} is maximal for the square. When $n=1$, every maximizer is a rectangle and every rectangle is a maximizer.

Among ellipses, the quantity \eqref{DNquantity} is maximal for the disk. When $n=1$, every maximizer is a disk.

The normalized Neumann eigenvalue sum
\[
(\mu_2 + \dots + \mu_n) \frac{A^3}{I}
\]
is maximal among triangles for the equilateral triangle, among parallelograms for the square, and among ellipses for the disk, for each $n \geq 2$. When $n=2$, every maximizer is an equilateral triangle, square, or disk, respectively.
\end{corollary}
The Neumann case with $n=1$ is not interesting, because the first eigenvalue equals $0$ for each domain.

\begin{remarks} \rm \

1. The method extends to linear images of regular $N$-gons for any $N$, but the most interesting cases are triangles and parallelograms ($N=3$ and $N=4$), as considered in the corollary.

\smallskip
2. For triangles, the moment of inertia can be calculated in terms of the side lengths $l_1,l_2,l_3$ as
\begin{equation} \label{trianglemoment}
I = \frac{A}{36} (l_1^2+l_2^2+l_3^2) .
\end{equation}
For a parallelogram with adjacent side lengths $l_1,l_2$, the moment of inertia equals
\begin{equation} \label{paramoment}
I = \frac{A}{12} (l_1^2+l_2^2) .
\end{equation}

\smallskip
3. The eigenvalues and eigenfunctions of the extremal domains (the equilateral triangle, square and disk) are \emph{not} used in our proofs. The eigenvalues are stated anyway in Appendix~\ref{eigenapp}, for reference. It is interesting to substitute them into Corollary~\ref{regularpolyDN} and obtain explicit estimates on eigenvalue sums. For example, for the Dirichlet fundamental tone ($n=1$) one obtains that
\begin{equation} \label{lambda1values}
\lambda_1 \frac{A^3}{I} \leq
\begin{cases}
12\pi^2 & \text{for triangles, with equality for equilaterals,} \\
12\pi^2 & \text{for parallelograms, with equality for rectangles,} \\
2 j_{0,1}^2 \pi^2 \simeq 11.5\pi^2 & \text{for ellipses, with equality for disks.} \\
\end{cases}
\end{equation}

All three inequalities were obtained by P\'{o}lya \cite{P52}, \cite[p.~308,328]{PS53}. The first inequality, for triangles, was rediscovered with a different proof by Freitas \cite[Theorem 1]{F06}. The second inequality was rediscovered for the special case of rhombi by Hooker and Protter \cite[\S5]{HP60} and for all parallelograms by Hersch \cite[formula (5)]{H66}, again with different proofs. These authors stated their results in terms of side lengths, as
\begin{equation} \label{lambda1again}
\lambda_1 \leq
\begin{cases}
(l_1^2+l_2^2+l_3^2)\pi^2/3A^2 & \text{for triangles,} \\
(l_1^2+l_2^2)\pi^2/A^2 & \text{for parallelograms.}
\end{cases}
\end{equation}
These inequalities are equivalent to P\'{o}lya's by formulas \eqref{trianglemoment} and \eqref{paramoment} for the moment of inertia.

For the first nonzero Neumann eigenvalue we find from Corollary~\ref{regularpolyDN} and Appendix~\ref{eigenapp} that
\begin{equation} \label{neumanntone}
\mu_2 \frac{A^3}{I} \leq
\begin{cases}
4\pi^2 & \text{for triangles, with equality for equilaterals,} \\
6\pi^2 & \text{for parallelograms, with equality for squares,} \\
2 (j^\prime_{1,1})^2 \pi^2 \simeq 6.8\pi^2 & \text{for ellipses, with equality for disks.} \\
\end{cases}
\end{equation}
These inequalities too can be stated in terms of side lengths. For stronger inequalities on $\mu_2$, see Section~\ref{literature}.

For $n=3$, the Corollary says $(\mu_2+\mu_3)A^3/I$ is maximal for the equilateral triangle. This result was proved recently by the authors using a different method with explicit trial functions \cite[Theorem 3.5]{LS09}.

\smallskip
4. Corollary~\ref{regularpolyDN} becomes false when applied to individual eigenvalues instead of eigenvalue sums. For example, $\lambda_3 A^3/I$ is not maximal for the square among rectangles: to the contrary, it is locally minimal. The underlying reason is that a square has a double eigenvalue $\lambda_2=\lambda_3$ that ``splits'' when the square is deformed into a rectangle of the same area; the second eigenvalue decreases and the third increases, while the moment of inertia varies only at second order.
\end{remarks}

\subsection*{Robin eigenvalues} In the next theorem we normalize the Robin parameter in terms of $T^{-1}$, in order to obtain a scale invariant expression.
\begin{theorem} \label{2dimR}
If $D$ has rotational symmetry of order greater than or equal to $3$, then
\[
\left. (\rho_1 + \dots + \rho_n) \frac{A^3}{I} \right|_{\sigma \lVert \, T^{-1} \rVert_{HS}/\sqrt{2} \, , \, T(D)} \, \leq \,
\left. (\rho_1 + \dots + \rho_n) \frac{A^3}{I} \right|_{\, \sigma \, , \, D}
\]
for each $n \geq 1$ and each invertible linear transformation $T$ of the plane.

Equality holds for the first Robin eigenvalue ($n=1$) if and only if $T$ is a scalar multiple of an orthogonal matrix.
\end{theorem}
The subscript ``$\sigma , D$'' on the right side of the inequality specifies the domain where the eigenvalues and geometric quantities are to be evaluated, and also the value of the Robin parameter to be used. The subscript on the left side of the inequality similarly specifies the domain $T(D)$ and the Robin parameter $\sigma \lVert T^{-1} \rVert_{HS}/\sqrt{2}$ to be used there.
\begin{corollary} \label{regularpolyR}
Fix the Robin parameter $\sigma>0$. Among all triangles of the same area, the quantity
\[
(\rho_1 + \dots + \rho_n) \frac{A^3}{I}
\]
is maximal for the equilateral triangle. When $n=1$, every maximizer is equilateral.

Analogous results hold among parallelograms and ellipses, with squares and disks being the maximizers, respectively.
\end{corollary}

\subsection*{Schr\"{o}dinger eigenvalues}
Consider the Schr\"{o}dinger eigenvalue problem
\[
-h \Delta u +Wu = {\mathcal E} u
\]
in the plane, with Planck constant $h>0$ and real-valued potential $W \in L^\infty_{loc}(\R^2)$ that tends to $+\infty$ as $|x| \to \infty$. The spectrum is discrete \cite[Theorem XIII.67]{RS78}, and the eigenvalues ${\mathcal E}_j$ are characterized in the usual way by the Rayleigh quotient
\[
R[u] = \frac{h \int_{\R^2} |\nabla u|^2 \, dx + \int_{\R^2} Wu^2 \, dx}{\int_{\R^2} u^2 \, dx} , \qquad u \in H^1(\R^2) \cap L^2(W) .
\]
Here $L^2(W)$ denotes the weighted space with measure $|W| \, dx$.

Once more we show that a rotationally symmetric situation maximizes the sum of eigenvalues.
\begin{theorem} \label{2dimS}
If $W$ has rotational symmetry of order greater than or equal to $3$, then
\[
({\mathcal E}_1 + \dots + {\mathcal E}_n) \big|_{2h / \lVert T^{-1} \rVert_{HS}^2,W \circ T^{-1}} \leq
({\mathcal E}_1 + \dots + {\mathcal E}_n) \big|_{h,W}
\]
for each $n \geq 1$ and each invertible linear transformation $T$ of the plane.

When $n=1$, equality holds if and only if $T$ is a scalar multiple of an orthogonal matrix.
\end{theorem}
The subscript ``$h,W$'' on the right side of the inequality specifies the potential to be used and the value of the Planck constant, and similarly for the subscript on the left side of the inequality.

This Sch\"{o}dinger result formally implies the Dirichlet result in Theorem~\ref{2dimDN}, by taking $W=0$ on $D$ and $W=+\infty$ off $D$, and choosing $h=\lVert T^{-1} \rVert_{HS}^2/2$.

\subsection*{More general quadrilaterals}
Among quadrilaterals we have so far handled only the special class of parallelograms. Now we show how to handle a larger class of quadrilaterals having two ``halves'' of equal area.

To construct such domains, first write the upper and lower halfplanes as
\[
\R^2_+ = \{ (x_1,x_2) : x_2 > 0 \} , \qquad \R^2_- = \{ (x_1,x_2) : x_2 < 0 \} .
\]
Choose two linear transformations $T_+$ and $T_-$ that agree on the $x_1$-axis, with $T_\pm$ mapping $\R^2_\pm$ onto itself. Then the map
\[
T(x) =
\begin{cases}
T_+ x & \text{if\ } x \in \overline{\R^2_+}, \\
T_- x & \text{if\ } x \in \overline{\R^2_-},
\end{cases}
\]
defines a piecewise linear homeomorphism of the plane mapping the upper and lower halfplanes to themselves. Assume also $\det T_+ = \det T_-$, so that $T$ distorts areas by the same factor in the upper and lower halfplanes.

We will not need explicit formulas for the linear transformations $T_+$ and $T_-$, but for the sake of concreteness we present them anyway:
\[
T_\pm =
\begin{pmatrix}
a & c_\pm \\ 0 & b
\end{pmatrix}
\]
where $a \neq 0, b>0$ and $c_\pm \in \R$.

Let $D$ be the square with vertices at $(\pm 1,0),(0,\pm 1)$. Our goal is to show that this square maximizes eigenvalue sums among quadrilaterals of the form $E=T(D)$ . These quadrilaterals have two vertices on the $x_1$-axis and have upper and lower halves of equal area.

Write
\[
I_0(E)_=\int_E |x|^2 \, dx
\]
for the moment of inertia about the origin, for a domain $E$.
\begin{theorem}[Quadrilaterals with equal-area halves] \label{quadDN}
Let $D$ be the square with vertices at $(\pm 1,0),(0,\pm 1)$. Then for every map $T$ constructed as above,
\[
\left. (\lambda_1 + \dots + \lambda_n) \frac{A^3}{I_0} \right|_{T(D)} \leq \left. (\lambda_1 + \dots + \lambda_n) \frac{A^3}{I_0}  \right|_D
\]
for each $n \geq 1$.

The inequality holds also for Neumann eigenvalues.
\end{theorem}
The moment $I_0$ is generally greater than the moment of inertia $I$ for the domain $T(D)$, because the centroid of $T(D)$ need not be at the origin. For example, if both $T_+$ and $T_-$ are shear transformations towards the right, then the centroid of $T(D)$ will lie on the positive $x_1$-axis, to the right of the origin. Meanwhile, the centroid of the rotationally symmetric domain $D$ will always lie at the origin, so that $I(D)=I_0(D)$.

We conjecture that Theorem~\ref{quadDN} can be strengthened to use $I$ instead of $I_0$.

For the first eigenvalue ($n=1$), Freitas and Siudeja  \cite{FS10} showed recently with a computer-assisted proof that
\[
\lambda_1 \, \frac{A^2}{l_1^2+l_2^2+l_3^2+l_4^2}
\]
is maximal for rectangles among \emph{all} quadrilaterals, not just among quadrilaterals with halves of equal area. For parallelograms, this result and Theorem~\ref{quadDN} give the same information (see formula \eqref{lambda1again}). For general quadrilaterals we cannot easily compare the two results, because the relationship between the sum of squares of side lengths and the moment of inertia is unclear.

\section{\bf Open problems for general convex domains} \label{open}
For the Dirichlet fundamental tone we raise:
\begin{conjecture} \label{conj1}
Suppose $\Omega$ is a bounded convex plane domain. Then
\[
\frac{9}{2} \pi^2 < \left. \lambda_1 \frac{A^3}{I} \right|_\Omega \leq 12\pi^2
\]
with equality on the right for equilateral triangles and all rectangles, and asymptotic equality on the left for degenerate acute isosceles triangles and sectors.
\end{conjecture}
The convexity assumption is necessary on the right side of the conjecture because otherwise one could drive the eigenvalue to infinity without affecting the area or moment of inertia, by removing sets of measure zero (such as curves) from the domain.

The maximizer cannot be the disk in the last conjecture because triangles and rectangles yield a larger value, as we observed already in \eqref{lambda1values}.

As evidence for the conjecture, we note that $\lambda_1 A^3/I$ is bounded above and below on convex domains by an Inclusion Lemma, as was shown by P\'{o}lya and Szeg\H{o} \cite[\S1.19,5.11b]{PS51}. They further evaluated $\lambda_1 A^3/I$ for a variety of triangles, sectors, degenerate ellipses and degenerate sectors \cite[p.~267]{PS51}. Asymptotic expansions can be obtained also for degenerate triangles \cite{F07}. We examine the family of isosceles triangles in Figure~\ref{fig:degtriangle}. All this evidence is consistent with Conjecture~\ref{conj1}.
\begin{figure}[ht]
  \begin{center}
\begin{tikzpicture}[xscale=3,yscale=0.05,smooth]
      \draw[<->] (0,135) node [left] {$\lambda_1A^3/I$} |- (3.3,35) node [pos=1,below] {$\alpha$};
      \draw (0,35) node[below] {$0$};
%      \draw (pi/6,35) node[below] {$\pi/6$};
      \draw[dotted] (pi/3,118.4367) -- (pi/3,35) node[below] {$\pi/3$};
      \draw (pi/2,35) node[below] {$\pi/2$};
      \draw (pi,35) node[below] {$\pi$};
%      \draw (2*pi/3,35) node[below] {$2\pi/3$};
      \draw [mark=|, mark size=10] plot coordinates{(0,35) (pi/3,35) (pi/2,35) (pi,35)};
      \draw[blue] plot coordinates {
      (0,4.5*pi^2)
( 0.004,46.3976)( 0.0156,49.4640)( 0.0281,51.9762)( 0.0405,54.2100)( 0.0530,56.2766)
( 0.0654,58.2275)( 0.0884,61.5993)( 0.1113,64.7672)( 0.1342,67.7790)( 0.1571,70.6616)( 0.1800,73.4310)( 0.2029,76.0971)( 0.2258,78.6664)( 0.2487,81.1431)( 0.2716,83.5298)( 0.2945,85.8284)( 0.3174,88.0400)( 0.3403,90.1654)( 0.3632,92.2052)( 0.3862,94.1598)( 0.4091,96.0297)( 0.4320,97.8152)( 0.4549,99.5169)( 0.4778,101.1354)( 0.5007,102.6714)( 0.5236,104.1257)
( 0.5760,107.1483)( 0.6283,109.7632)( 0.6807,111.9873)( 0.7330,113.8400)( 0.7854,115.3431)( 0.8378,116.5196)( 0.8901,117.3933)( 0.9425,117.9883)( 0.9948,118.3284)( 1.0472,118.4367)( 1.0996,118.3356)( 1.1519,118.0462)( 1.2043,117.5882)( 1.2566,116.9802)( 1.3090,116.2391)( 1.3614,115.3806)( 1.4137,114.4187)( 1.4661,113.3663)( 1.5184,112.2349)( 1.5708,111.0348)( 1.6232,109.7753)( 1.6755,108.4643)( 1.7279,107.1091)( 1.7802,105.7159)( 1.8326,104.2901)( 1.8850,102.8364)( 1.9373,101.3588)( 1.9897,99.8607)( 2.0420,98.3448)( 2.0944,96.8135)
( 2.1991,93.7080)( 2.3038,90.3)( 2.4086,86.8)( 2.5133,83)%( 2.6180,80.80)
      (pi,6*pi^2)
      };
      \draw[dotted] ( 1.0472,118.4367) -- (0,118.4367) node[left] {$12\pi^2$};
      \draw[dotted] ( pi,6*pi^2) -- (0,6*pi^2) node[left] {$6\pi^2$};
      \draw (0,4.5*pi^2) node [left] {$9\pi^2/2$};
      \draw [mark=-, mark size=0.3] plot coordinates{(0,118.4367) (0,6*pi^2) (0,4.5*pi^2)};
\end{tikzpicture}  \end{center}
  \caption{Numerical plot of the normalized Dirichlet fundamental tone $\lambda_1 A^3/I$ for isosceles triangles of aperture $\alpha \in (0,\pi)$. The maximizer is equilateral ($\alpha=\pi/3$), and the minimizer is degenerate acute ($\alpha \to 0$).}
  \label{fig:degtriangle}
\end{figure}
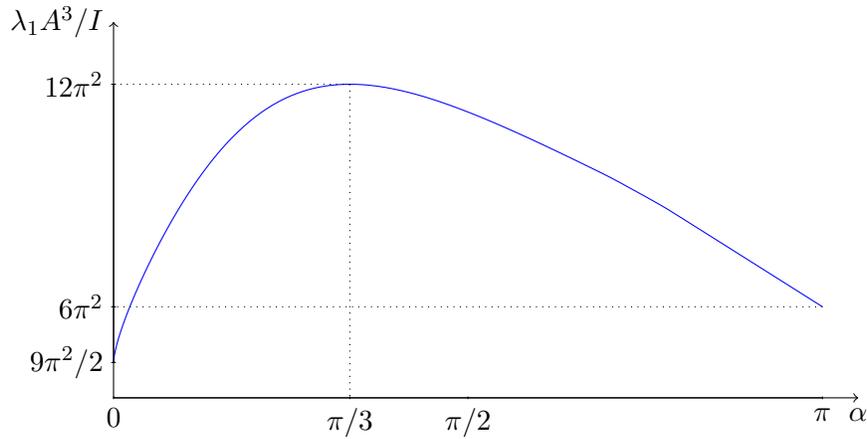

For the first nonzero Neumann eigenvalue, we know $\mu_2 A^3/I$ is definitely maximal for the disk among all bounded domains, by an inequality of Szeg\H{o} and Weinberger (see Section~\ref{literature}). This quantity has no minimizer because it approaches zero for a degenerate rectangle.

Now consider sums of eigenvalues.
\begin{conjecture} \label{2dimDNconj}
Suppose $\Omega$ is a bounded convex plane domain. Then for the Neumann eigenvalues,
\[
\left. (\mu_2 + \dots + \mu_n) \frac{A^3}{I} \right|_{\Omega}
\]
is maximal when $\Omega$ is a disk, for each $n \geq 2$.
\end{conjecture}
The conjecture is true for the special case of ellipses by Corollary~\ref{regularpolyDN}.

For Dirichlet eigenvalues, the conjecture fails because the square gives a larger value  than the disk for $(\lambda_1+\cdots+\lambda_n)A^3/I$ when $n=1,2,3,5,6,9,10,12$; the disk does give a larger value for all other $n \leq 50$, and we suspect for $n>50$ as well.

\section{\bf Consequences of symmetry: tight frames, and moment matrices} \label{symmetry}

In this section we recall the \emph{tight frame} property of rotationally symmetric systems of vectors, and develop a moment of inertia formula for the linear image of a rotationally symmetric domain. These elementary consequences of symmetry will be used in proving Theorem~\ref{2dimDN}.

\subsection*{Tight frames}
Let $N \geq 3$ and write $U_m$ for the matrix representing rotation by angle $2\pi m/N$, for $m=1,\ldots,N$. For each nonzero $y \in \R^2$, the rotations generate a rotationally symmetric system $\{ U_1 y,\ldots, U_N y \}$ in the plane. For example, the system consists of the $N$th roots of unity when $y=\left( \begin{smallmatrix} 1 \\ 0 \end{smallmatrix} \right)$.

We start with a well known Plancherel-type identity for such systems.
\begin{lemma} \label{2dimtightframe}
Let $N \geq 3$. For all column vectors $x,y \in \R^2$ one has
\[
\frac{1}{N} \sum_{m=1}^N |x \cdot (U_m y)|^2 = \frac{1}{2} |x|^2 |y|^2 .
\]
\end{lemma}
\begin{proof}[Proof of Lemma~\ref{2dimtightframe}]
We may suppose $x$ and $y$ have length $1$ and lie at angles $\theta$ and $\phi$ to the positive horizontal axis, respectively. Then
\begin{align*}
\frac{1}{N} \sum_{m=1}^N |x \cdot (U_m y)|^2
& = \frac{1}{N} \sum_{m=1}^N \cos^2(\theta-\phi-2\pi m/N) \\
& = \frac{1}{2} + \frac{1}{2N} \sum_{m=1}^N \cos2(\theta-\phi-2\pi m/N) \\
& = \frac{1}{2} + \frac{1}{2N} \Real \Big( e^{i2(\theta-\phi)} \sum_{m=1}^N (e^{-i4\pi/N})^m \Big) \\
& = \frac{1}{2}
\end{align*}
as desired. The assumption $N \geq 3$ ensures that $e^{-i4\pi/N} \neq 1$ when summing the geometric series, in the last step.
\end{proof}
Figure~\ref{fig:benzframe} illustrates the lemma for $N=3$: it shows the projection formula $\sum_{m=1}^3 (x \cdot U_m y) U_m y = \frac{3}{2} x$ for a typical $x \in \R^2$, where $y=\left( \begin{smallmatrix}0 \\ 1 \end{smallmatrix} \right)$ is the vertical unit vector and $U_m$ denotes rotation by $2\pi m/3$. Dotting the projection formula with $x$ yields the Parseval-type identity in Lemma~\protect\ref{2dimtightframe}.

\newcommand{\benz}[1]{
\begin{tikzpicture}[scale=3]
  \clip (-1.1,-0.8) rectangle (1.1,1.4);
  \draw[black,->] (0,0) -- (90:1) coordinate (a)
  node [pos=1,left] {$y$};
  \draw[black,->] (0,0) -- (210:1) coordinate (b);
  \draw[black,->] (0,0) -- (330:1) coordinate (c);
  \draw[very thick,red,->] (0,0) -- (#1:1) coordinate (d)
  node [pos=1,above] {$x$};
  \draw[dashed,blue] (d) -- ($(0,0)!(d)!(a)$) coordinate (e);
  \draw[dashed,blue] (d) -- ($(0,0)!(d)!(b)$) coordinate (f);
  \draw[dashed,blue] (d) -- ($(0,0)!(d)!(c)$) coordinate (g);
  \draw[thick,blue,->] (0,0) -- (e);
  \draw[thick,blue,->] (0,0) -- (f);
  \draw[thick,blue,->] (0,0) -- (g);
\end{tikzpicture}
\hspace{\fill}
\begin{tikzpicture}[scale=3]
  \clip (-0.1,-0.8) rectangle (1.3,1.4);
  \draw[very thick,red,->] (0,0) -- (d);
%  \draw[dotted,blue] (d) -- (e);
% \draw[dotted,blue] (d) -- (f);
%  \draw[dotted,blue] (d) -- (g);
  \draw[thick, blue,->] (0,0) -- (e);
  \draw[thick, blue,->] (0,0) -- (f);
  \draw[thick, blue,->] (0,0) -- (g);
  \draw[dashed,green!50!black,-] (g) -- ++(f) -- ++(e) coordinate (h);
  \draw[dashed,green!50!black,-] (g) -- ++(e) -- ++(f);
  \draw[dashed,green!50!black,-] (f) -- ++(e) -- ++(g);
  \draw[dashed,green!50!black,-] (f) -- ++(g);
  \draw[dashed,green!50!black,-] (e) -- ++(f);
  \draw[dashed,green!50!black,-] (e) -- ++(g);
  \draw[green!50!black,->] (0,0) -- (h)
  node [pos=1,above] {$\frac{3}{2}x$};
\end{tikzpicture}
}

\begin{figure}[th]
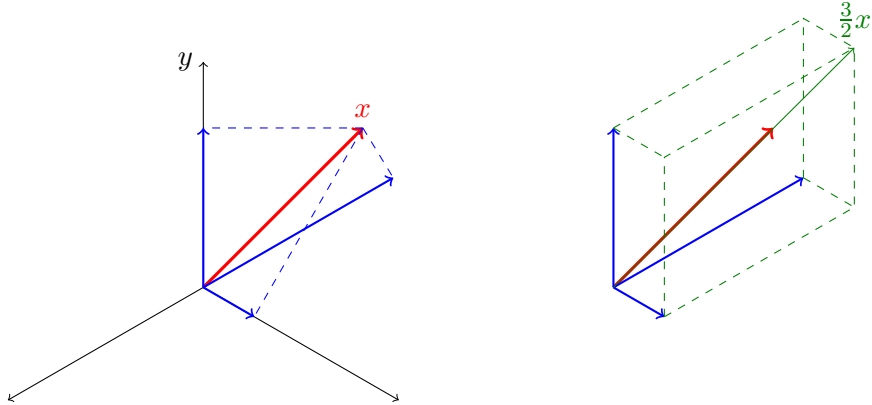

  \begin{center}
    %\benz{15}
    \benz{45}
    %\benz{60}
    %\benz{75}
  \end{center}
  \caption{The ``Mercedes--Benz'' tight frame ($N=3$) in the plane.}
  \label{fig:benzframe}
\end{figure}

The lemma says that the rotationally symmetric system $\{ U_1 y,\ldots, U_N y \}$ forms a \emph{tight frame}. Readers who want to learn about frames and their applications in Hilbert spaces may consult the monograph by Christensen \cite{Ch03} or the text by Han \emph{et al.}\ \cite{H07}.

\smallskip
We next deduce a tight frame identity in which the vector $y$ is replaced by a matrix.
\begin{lemma} \label{tightframematrix}
Let $N \geq 3,  K \geq 1$. For all row vectors $x \in \R^2$ and all $2 \times K$ real matrices $Y$ one has
\[
\frac{1}{N} \sum_{m=1}^N |x U_m Y|^2 = \frac{1}{2} |x|^2 \, \lVert Y \rVert_{HS}^2 .
\]
\end{lemma}
\begin{proof}[Proof of Lemma~\ref{tightframematrix}]
Write $y_1,\ldots,y_K$ for the column vectors of $Y$, so that $|x U_m Y|^2
= \sum_{k=1}^K |x U_m y_k|^2$. Now apply Lemma~\ref{2dimtightframe}. 
\end{proof}

\subsection*{Hilbert--Schmidt norms and moment of inertia}
When proving Corollary~\ref{regularpolyDN}, we will need to evaluate the Hilbert--Schmidt norm of $T^{-1}$ in terms of moment of inertia and area.
\begin{lemma} \label{2dimHSnorm}
If the bounded plane domain $D$ has rotational symmetry of order $N \geq 3$, and $T$ is an invertible $2 \times 2$ matrix, then
\[
\frac{1}{2} \lVert T^{-1} \rVert_{HS}^2 = \frac{I}{A^3}(TD) \Big/ \frac{I}{A^3}(D) .
\]
\end{lemma}
\begin{proof}[Proof of Lemma~\ref{2dimHSnorm}]
The centroid of $D$ lies at the origin, in view of the rotational symmetry of $D$. Thus the centroid of $TD$ also lies at the origin.

The moment matrix of $D$ is defined to be $M(D)=[\int_D x_j x_k \, dx]_{j,k}$. We show it equals a scalar multiple of the identity, as follows. Let $U$ denote the matrix for rotation by $2\pi/N$. The rotational invariance of $D$ under $U$ implies that $M(D)=UM(D)U^\dagger$, so that if $x$ is an eigenvector of $M(D)$ then so is $Ux$, with the same eigenvalue. Since $x$ and $Ux$ span $\R^2$ (using here that $N \geq 3$), we conclude every vector in $\R^2$ is an eigenvector with that same eigenvalue. Thus $M(D)$ is a multiple of the identity.

In particular, the diagonal entries in $M(D)$ are equal. Since they sum to the moment of inertia $I(D)$, we have
\begin{equation} \label{momentidentity}
M(D) = \frac{1}{2} I(D)
\begin{pmatrix}
1 & 0 \\ 0 & 1
\end{pmatrix} .
\end{equation}

The moment of inertia of $T(D)$ can now be computed as
\begin{align}
I(TD)
& = \tr M(TD) \notag \\
& = \tr T M(D) T^\dagger |\det T| \notag \\
& =\frac{1}{2} I(D) \, \big( \tr T T^\dagger \big) |\det T| \qquad \text{by \eqref{momentidentity}} \notag \\
& = \frac{1}{2} I(D) \, \lVert T \rVert_{HS}^2 \, |\det T| . \label{momenteq1}
\end{align}
This formula gives us the Hilbert--Schmidt norm of $T$, whereas we want the Hilbert--Schmidt norm of $T^{-1}$. Fortunately, the two are related, with
\begin{equation} \label{inverseHS}
\lVert T^{-1} \rVert_{HS}^2 = \lVert T \rVert_{HS}^2 / |\det T|^2
\end{equation}
by the explicit formula for $T^{-1}$ in terms of the matrix entries, in two dimensions. Hence
\begin{equation} \label{2dimHSinverse}
I(TD)
= \frac{1}{2} I(D) \, \lVert T^{-1} \rVert_{HS}^2 \, |\det T|^3 ,
\end{equation}
from which the lemma follows easily.
\end{proof}

An interesting consequence of the last lemma is that the moment of inertia of a linear image of a rotationally symmetric domain equals the moment of inertia of its \emph{inverse} image, after normalizing by the area.
\begin{lemma} \label{2diminvariance}
If the bounded plane domain $D$ has rotational symmetry of order $N \geq 3$, and $T$ is an invertible $2 \times 2$ matrix, then
\[
\frac{I}{A^2}(TD) = \frac{I}{A^2}(T^{-1}D) .
\]
\end{lemma}
\begin{proof}[Proof of Lemma~\ref{2diminvariance}]
By \eqref{momenteq1}, and then using \eqref{2dimHSinverse} with $T$ replaced by $T^{-1}$, we find
\[
\frac{I(TD)}{A(TD)^2}
= \frac{1}{2} \frac{I(D)}{A(D)^2} \frac{\lVert T \rVert_{HS}^2}{|\det T|} = \frac{I(T^{-1} D)}{A(T^{-1} D)^2} .
\]
\end{proof}
The lemma holds also with $T^{-\dagger} D$ instead of $T^{-1}D$, since $T^\dagger$ and $T$ have the same Hilbert--Schmidt norm and determinant.

\section{\bf Proof of Theorem~\ref{2dimDN}} \label{2dimDN_proof}

We prove the Dirichlet case of the theorem. The idea is to construct trial functions on the domain $T(D)$ by linearly transplanting eigenfunctions of $D$, and then to average with respect to the rotations of $D$. The Neumann proof is identical, except using Neumann eigenfunctions.

Let $u_1,u_2,u_3,\ldots$ be orthonormal eigenfunctions on $D$ corresponding to the Dirichlet eigenvalues $\lambda_1,\lambda_2,\lambda_3,\ldots$. Consider an orthogonal matrix $U \in O(2)$ that fixes $D$, so that $U(D)=D$. Define trial functions
\[
v_j = u_j \circ U \circ T^{-1}
\]
on the domain $E=T(D)$, noting $v_j \in H^1_0(E)$ because $u_j \in H^1_0(D)$.

The functions $v_j$ are pairwise orthogonal, since
\[
\int_E v_j v_k \, dx = \int_D u_j u_k \, dx \cdot |\det TU^{-1}| = 0
\]
when $j \neq k$. Thus by the Rayleigh--Poincar\'{e} principle, we have
\begin{equation} \label{rayleighprinc}
\sum_{j=1}^n \lambda_j(E) \leq \sum_{j=1}^n \frac{\int_E |\nabla v_j|^2 \, dx}{\int_E v_j^2 \, dx} .
\end{equation}
For each function $v=v_j$ we evaluate the Rayleigh quotient as
\begin{align}
\frac{\int_E |\nabla v|^2 \, dx}{\int_E v^2 \, dx}
& = \frac{\int_D |(\nabla u)(x)UT^{-1}|^2 \, dx \cdot |\det TU^{-1}|}{\int_D u^2 \, dx \cdot |\det TU^{-1}|} \notag \\
& = \int_D |(\nabla u) UT^{-1}|^2 \, dx , \label{rayleigh}
\end{align}
where the gradient $\nabla u$ is regarded as a row vector. In the last line we used that $u=u_j$ is normalized in $L^2(D)$.

Since $D$ has $N$-fold rotational symmetry for some $N \geq 3$, we may choose $U$ to be the matrix $U_m$ representing rotation by angle $2\pi m/N$, for $m=1,\ldots,N$. By averaging \eqref{rayleighprinc} and \eqref{rayleigh} over these rotations we find
\begin{align*}
\sum_{j=1}^n \lambda_j(E)
& \leq \sum_{j=1}^n \int_D \Big\{ \frac{1}{N} \sum_{m=1}^N |(\nabla u_j) U_m T^{-1}|^2 \Big\} \, dx \\
& = \sum_{j=1}^n \int_D \Big\{ \frac{1}{2} |\nabla u_j|^2 \lVert T^{-1} \rVert_{HS}^2 \Big\} \, dx \qquad \text{by Lemma~\ref{tightframematrix}} \\
& = \frac{1}{2} \lVert T^{-1} \rVert_{HS}^2 \sum_{j=1}^n \lambda_j(D) ,
\end{align*}
which proves the inequality in Theorem~\ref{2dimDN}.

\subsubsection*{Equality statement for Dirichlet fundamental tone, $n=1$.}
Suppose equality holds in the theorem for the first Dirichlet eigenvalue. That is, suppose
\begin{equation} \label{fundeq}
\lambda_1 \big( T(D) \big) = \frac{1}{2} \lVert T^{-1} \rVert_{HS}^2 \, \lambda_1(D) .
\end{equation}

We reduce to $T$ being diagonal, as follows. The singular value decomposition of $T$ can be written $T=QRS$ where $Q$ and $S$ are orthogonal matrices with $\det S=1$ (so that $S$ is a rotation matrix) and  $R=\left( \begin{smallmatrix} r_1 & 0 \\ 0 & r_2 \end{smallmatrix} \right)$ is diagonal with $r_1 , r_2 > 0$. If $r_1=r_2$ then $T$ is a scalar multiple of an orthogonal matrix. So suppose from now on that $r_1 \neq r_2$.

Write $\widetilde{D}=S(D)$, so that $\widetilde{D}$ has rotational symmetry of order $N$. Note $\lambda_1(D)=\lambda_1(\widetilde{D})$, and that $T(D)=QR(\widetilde{D})$ so that
\[
\lambda_1 \big( T(D) \big) = \lambda_1 \big( R(\widetilde{D}) \big) .
\]
Also
\[
\lVert T^{-1} \rVert_{HS}^2 = \lVert R^{-1} \rVert_{HS}^2 = r_1^{-2} + r_2^{-2} .
\]
Hence equality in \eqref{fundeq} implies
\[
\lambda_1 \big( R(\widetilde{D}) \big)= \frac{1}{2} (r_1^{-2} + r_2^{-2}) \lambda_1(\widetilde{D}) ,
\]
which means that equality holds in \eqref{DNeq} for the domain $\widetilde{D}$ under the diagonal linear transformation $R$.

Write $u=u_1$ for a first Dirichlet eigenfunction on $\widetilde{D}$, so that
\begin{equation} \label{linear1}
u_{x_1 x_1} + u_{x_2 x_2} = -\lambda_1(\widetilde{D}) u .
\end{equation}
Inspecting the proof of the theorem, above, we see that one of the trial functions on $R(\widetilde{D})$ is $v=u \circ R^{-1}$, in other words $v(x_1,x_2)=u(x_1/r_1,x_2/r_2)$. Since equality holds in the Rayleigh principle \eqref{rayleighprinc} with $n=1$, we deduce that this trial function must actually be a first eigenfunction on $R(\widetilde{D})$. That is,
\[
\Delta v = -\lambda_1 \big( R(\widetilde{D}) \big) v ,
\]
which means
\begin{equation} \label{linear2}
r_1^{-2} u_{x_1 x_1} + r_2^{-2} u_{x_2 x_2} = -\frac{1}{2}(r_1^{-2}+r_2^{-2}) \lambda_1(\widetilde{D}) u .
\end{equation}
By solving the simultaneous linear equations \eqref{linear1} and \eqref{linear2} (which is possible since $r_1 \neq r_2$) we find that
\begin{equation} \label{eigen}
u_{x_1 x_1} = u_{x_2 x_2} = -\frac{1}{2} \lambda_1(\widetilde{D}) u .
\end{equation}
This last formula must apply also if we rotate $u$ through angle $2\pi/N$, because that rotate of $u$ was used in one of the trial functions in the proof of the theorem above. Hence the second directional derivative of $u$ in direction $\theta=2\pi/N$ must equal $-\frac{1}{2} \lambda_1(\widetilde{D}) u$. That second derivative is
\[
(\cos^2 \theta) u_{x_1 x_1} + 2(\cos \theta \sin \theta) u_{x_1 x_2} + (\sin^2 \theta) u_{x_2 x_2} ,
\]
which equals $-\frac{1}{2} \lambda_1(\widetilde{D}) u + \sin (2\theta) u_{x_1 x_2}$ by \eqref{eigen}. We conclude $\sin (2\theta) u_{x_1 x_2}=0$.

Suppose $N \neq 4$. Then $\sin (2\theta) = \sin(4\pi/N) \neq 0$, and so $u_{x_1 x_2}=0$ in $\widetilde{D}$. Then $u=F_1(x_1)+F_2(x_2)$ for some functions $F_1$ and $F_2$, and substituting this formula into \eqref{eigen} gives that $F_1^{\prime \prime}(x_1)=-\frac{1}{2} \lambda_1(\widetilde{D})[F_1(x_1)+F_2(x_2)]$. Taking the $x_2$ derivative shows that $F_2$ is constant. Similarly $F_1$ is constant, and so $u$ is constant, an impossibility.

Therefore $N=4$. Write $\omega=\sqrt{\lambda_1(\widetilde{D})/2}$. The equations \eqref{eigen} say $u_{x_1 x_1} = u_{x_2 x_2} = - \omega^2 u$, and so
\begin{align*}
u(x_1,x_2) & = A \cos(\omega x_1) \cos(\omega x_2) + B \sin(\omega x_1) \sin(\omega x_2) \\
& \quad + C \cos(\omega x_1) \sin(\omega x_2) + D_* \sin(\omega x_1) \cos(\omega x_2)
\end{align*}
on $\widetilde{D}$, for some constants $A,B,C,D_*$. The $4$-fold rotational symmetry of the domain further implies that each of the four terms
\begin{align}
A & \cos(\omega x_1) \cos(\omega x_2) \label{term1} \\
B & \sin(\omega x_1) \sin(\omega x_2) \label{term2} \\
(C^2+D_*^2) & \cos(\omega x_1) \sin(\omega x_2) \label{term3} \\
(C^2+D_*^2) & \sin(\omega x_1) \cos(\omega x_2) \label{term4}
\end{align}
is by itself a Dirichlet mode for $\widetilde{D}$ with eigenvalue $2\omega^2 = \lambda_1(\widetilde{D})$, or else is identically zero, as we will now show. First, by adding and subtracting $u(x_1,x_2)$ and $u(-x_1,-x_2)$ (its rotation by $\pi$) we find that the functions
\[
f(x_1,x_2) = A \cos(\omega x_1) \cos(\omega x_2) + B \sin(\omega x_1) \sin(\omega x_2)
\]
and
\[
g(x_1,x_2) = C \cos(\omega x_1) \sin(\omega x_2) + D_* \sin(\omega x_1) \cos(\omega x_2)
\]
are each eigenfunctions on $\widetilde{D}$ (or else are identically zero). By adding and subtracting $f(x_1,x_2)$ and $f(-x_2,x_1)$ (rotation by $\pi/2$) we find that \eqref{term1} and \eqref{term2} are each eigenfunctions (or else are identically zero). By considering $Cg(x_1,x_2)-D_*g(-x_2,x_1)$ and $D_*g(x_1,x_2)+Cg(-x_2,x_1)$ we learn that \eqref{term3} and \eqref{term4} are each eigenfunctions (or else are identically zero).

The fundamental Dirichlet mode does not change sign. The nodal domains for each of the functions \eqref{term1}--\eqref{term4} are squares, and so $\widetilde{D}$ must lie within one of those squares. For \eqref{term2}, \eqref{term3} and \eqref{term4}, rotation by angle $\pi$ maps each nodal square to a completely disjoint square, which means that $\widetilde{D}$ cannot have $2$-fold rotational symmetry, let alone $4$-fold symmetry. Hence \eqref{term2}, \eqref{term3} and \eqref{term4} must not be eigenfunctions, and so necessarily $B=C=D_*=0$. Thus the eigenfunction is \eqref{term1}. Taking $A=1$, we have
\[
u = \cos(\omega x_1) \cos(\omega x_2) .
\]
Rotation by $\pi$ rules out every nodal square except the one centered at the origin, which is $(-\pi/2\omega,\pi/2\omega)^2$. Hence $\widetilde{D}$ is contained in this square.

The square has first Dirichlet eigenvalue $2\omega^2$, which equals $\lambda_1(\widetilde{D})$. Thus $\widetilde{D}$ must fill the whole square (except perhaps omitting a set of capacity zero, which does not affect the fundamental tone \cite{GZ94}). Then $R(\widetilde{D})$ is a rectangle, and $D=S^{-1}(\widetilde{D})$ is a square and $T(D)=QR(\widetilde{D})$ is a rectangle. This completes the proof of the ``only if'' part of the proof of the equality statement.

For the ``if'' part of the equality statement, suppose $D$ is a square and $T(D)$ is a rectangle (possibly with sets of capacity zero removed). By rotating and reflecting $D$ and $T(D)$ suitably, we can suppose they have sides parallel to the coordinate axes and that $T=\left( \begin{smallmatrix} r_1 & 0 \\ 0 & r_2 \end{smallmatrix} \right)$ for some $r_1,r_2>0$. Writing $L$ for the side length of the square, we have
\[
\lambda_1(D) = 2(\pi/L)^2 , \qquad \lambda_1 \big( T(D) \big) = (\pi/r_1 L)^2 + (\pi/r_2 L)^2 = \frac{1}{2} \lVert T^{-1} \rVert_{HS}^2 \lambda_1(D) ,
\]
so that equality holds in \eqref{DNeq} with $n=1$.

\subsubsection*{Equality statement for second Neumann eigenvalue, $n=2$.}
Suppose equality holds in the theorem for the second Neumann eigenvalue. Most of the preceding argument in the Dirichlet equality case applies without change, simply replacing $\lambda_1$ with $\mu_2$ and the Dirichlet eigenfunction $u_1$ with the Neumann eigenfunction $u_2$, and replacing the word ``Dirichlet'' with ``Neumann''. The argument works because the first Neumann eigenvalue of $\widetilde{D}$ is zero, with constant eigenfunction $u_1 \equiv \text{const.}$, and so the trial function $v_1 = u_1 \circ U \circ R^{-1}$ is also constant and hence is a first eigenfunction on $R(\widetilde{D})$. Thus if equality holds in the Rayleigh principle \eqref{rayleighprinc} for $n=2$ then the trial function $v_2 = u_2 \circ U \circ R^{-1}$ is a second eigenfunction on $R(\widetilde{D})$.

The significant difference from the Dirichlet proof begins at the sentence ``The fundamental Dirichlet mode does not change sign.'' The second Neumann eigenfunction $u=u_2$ does change sign on $\widetilde{D}$: it has exactly two nodal domains $\{ u>0 \}$ and $\{ u<0 \}$, each of which is connected. (We know the eigenfunction has at least two nodal domains because $u$ is orthogonal to the constant eigenfunction; it has at most two by Courant's nodal domain theorem \cite[p.~112]{B80}.)

Consider each of the four possible forms of $u$ in turn, namely \eqref{term1}--\eqref{term4}. Each one has square nodal domains, and the two nodal domains of $u$ in $\widetilde{D}$ must be subsets of such squares. Hence $\widetilde{D}$ intersects exactly two of the squares. At the same time, $\widetilde{D}$ has $4$-fold rotational symmetry. These requirements prevent \eqref{term1} from being an eigenfunction for $\widetilde{D}$, because if $\widetilde{D}$ intersected two of the nodal squares, then it would have to intersect at least five of them. Hence $A=0$. Similarly \eqref{term2} cannot be an eigenfunction, and so $B=0$.

Next we deal with \eqref{term3}. (The argument is similar for \eqref{term4}.) Suppose $C^2+D_*^2 >0$ in \eqref{term3}, so that we may take
\begin{equation} \label{almostthere}
u = \cos(\omega x_1) \sin(\omega x_2) .
\end{equation}
Then in order for $\widetilde{D}$ to intersect exactly two of the nodal squares, they must be the squares adjacent to the origin, so that
\begin{equation} \label{squaredomain}
\widetilde{D} \subset (-\pi/2\omega,\pi/2\omega) \times (-\pi/\omega,\pi/\omega).
\end{equation}
We will deduce a contradiction below, so that necessarily $C^2+D_*^2=0$. Hence none of the functions \eqref{term1}--\eqref{term4} is an eigenfunction, and so the case $N=4$ cannot occur. Therefore the only way for equality to hold is to have $r_1=r_2$, so that $T$ is a scalar multiple of an orthogonal matrix.

To obtain the desired contradiction, we will examine how the Neumann boundary condition is affected by the linear transformation. Since the domain $\widetilde{D}$ has Lipschitz boundary, there exists an outward normal vector $(n_1,n_2)$ at almost every boundary point (with respect to arclength measure). At each such point $(x_1,x_2) \in \partial \widetilde{D}$, we know $u$ satisfies the Neumann (or natural) boundary condition
\[
0 = \nabla u \cdot (n_1,n_2) = u_{x_1}n_1 + u_{x_2}n_2 ;
\]
here we used that $u$ as defined by \eqref{almostthere} is globally smooth.
Further, the point $(r_1 x_1, r_2 x_2) \in \partial R(\widetilde{D})$ has an outward normal vector $(n_1/r_1,n_2/r_2)$. Since $v(x_1,x_2)=u(x_1/r_1,x_2/r_2)$ is a smooth eigenfunction on the closure of $R(\widetilde{D})$, it satisfies the Neumann boundary condition:
\[
0 = \nabla v \cdot (n_1/r_1,n_2/r_2) = u_{x_1}n_1/r_1^2 + u_{x_2}n_2/r_2^2 .
\]
Recalling that $r_1 \neq r_2$, these simultaneous equations imply
\[
u_{x_1} n_1 = 0 \qquad \text{and} \qquad u_{x_2}n_2 = 0 .
\]
Hence either $u_{x_1}=0$ or $u_{x_2}=0$, a.e.\ on $\partial \widetilde{D}$. Recalling the formula \eqref{almostthere} for $u$ and the constraint \eqref{squaredomain} on $\widetilde{D}$, we deduce that almost every boundary point is contained in the lines $\{ x_1 = 0, \pm \pi/2\omega \}, \{ x_2 = 0, \pm \pi/2\omega, \pm \pi/\omega \}$. Furthermore, we can rule out the vertical lines $\{ x_1 = \pm \pi/2\omega \}$ because on those lines $u_{x_1} \neq 0$ and so $n_1 = 0$, which means the normal would be vertical and the tangent horizontal, so that the boundary would depart the given vertical lines. Similarly we rule out the horizontal lines $\{ x_2 = 0, \pm \pi/\omega \}$. Hence the boundary of $\widetilde{D}$ must lie in the union of the lines $\{ x_1 = 0 \}, \{ x_2 = \pm \pi/2\omega \}$. Since these lines fail to bound a domain, we have arrived at a contradiction, as desired.

\subsection*{Why did P\'{o}lya not prove our theorem?} P\'{o}lya proved Theorem~\ref{2dimDN} for the first Dirichlet eigenvalue $\lambda_1$, except that he proved no equality statement. His result appeared in \cite{P52} and its proof in \cite[Chapter IV]{PS53}.

Why did he not prove the theorem for sums of eigenvalues, or for Neumann eigenvalues, as we do in this paper? Or for higher dimensions as we do in a forthcoming paper \cite{LS11c}?

A possible reason is that our method is subtly different from P\'{o}lya's. We use rotational symmetry at a later stage in the argument. This delay permits us to handle more than just the first eigenvalue, and to handle Neumann eigenvalues too. Let us explain in more detail. P\'{o}lya began by using rotational symmetry of the domain to obtain rotational symmetry of the fundamental Dirichlet eigenfunction $u_1$: he observed that the rotate of $u_1$ is itself a  positive eigenfunction and so must equal $u_1$. Then P\'{o}lya deduced that
\[
\int_D \left( \frac{\partial u_1}{\partial x_1} \right)^{\! \! 2} dx = \int_D \left( \frac{\partial u_1}{\partial x_2} \right)^{\! \! 2} dx = \frac{1}{2} \int_D |\nabla u_1|^2 \, dx , \qquad \int_D \frac{\partial u_1}{\partial x_1} \frac{\partial u_1}{\partial x_2} \, dx = 0 .
\]
Hence his linearly transplanted trial function $u_1 \circ T^{-1}$ has Rayleigh quotient
\[
\frac{\int_E |\nabla (u_1 \circ T^{-1})|^2 \, dx}{\int_E (u_1 \circ T^{-1})^2 \, dx} = \frac{\int_D |(\nabla u_1)(x)T^{-1}|^2 \, dx}{\int_D u_1^2 \, dx} = \frac{1}{2} \lVert T^{-1} \rVert_{HS}^2 \frac{\int_D |\nabla u_1|^2 \, dx}{\int_D u_1^2 \, dx}
\]
as desired.

The difficulty when trying to extend P\'{o}lya's approach to sums of eigenvalues is that the higher eigenfunctions are usually \emph{not} symmetric under rotations, because of sign changes. The insight that permits us to prove Theorem~\ref{2dimDN} is that while the rotate of a higher eigenfunction need not equal itself, it must still be an eigenfunction with the same eigenvalue, and thus can still be used to generate trial functions by linear transplantation. Our proof uses the whole family of rotations to generate many trial functions, and then averages over the resulting family of inequalities. This approach applies (without change!) to the Neumann eigenvalues too.

\section{\bf Proofs of other results} \label{2dimproofs}

\subsection*{Proof of Corollary~\ref{regularpolyDN}}
Every triangle can be written (after translation) as the image under a linear transformation $T$ of an equilateral triangle centered at the origin . Hence the inequality for triangles in Corollary~\ref{regularpolyDN} follows from Theorem~\ref{2dimDN} and Lemma~\ref{2dimHSnorm}. The statements about parallelograms and ellipses are proved similarly.

\subsubsection*{Remark on Dirichlet maximizers when $n \geq 2$.} It is not clear how to determine all  maximizing domains for sums of eigenvalues beyond the first. For example, some (but not all) non-square rectangles can maximize $(\lambda_1+\cdots+\lambda_n)A^3/I$ when $n \geq 2$, as we now show. Consider a rectangle with side lengths $l_1,l_2$, so that the area is $A=l_1l_2$, the moment of inertia is $I=l_1l_2(l_1^2+l_2^2)/12$ and
\[
\frac{A^3}{I} = \frac{12}{l_1^{-2}+l_2^{-2}} .
\]
The fundamental tone is
\[
\lambda_1 = \pi^2(l_1^{-2}+l_2^{-2}) .
\]
Notice $\lambda_1 A^3/I=12\pi^2$ for every rectangle (not just for the square), as we have observed before. Thus every rectangle is a maximizer when $n=1$.

Now fix $n \geq 2$, and fix the side length $l_2$. For $l_1$ sufficiently large, we have eigenvalues $\lambda_j = \pi^2(j^2 l_1^{-2}+l_2^{-2})$ for $j=1,\ldots,n$, and so
\[
\lim_{l_1 \to \infty} (\lambda_1+\cdots+\lambda_n)\frac{A^3}{I} = 12 \pi^2 n .
\]
Meanwhile, the square satisfies
\[
(\lambda_1+\cdots+\lambda_n)\frac{A^3}{I} > n \lambda_1 \frac{A^3}{I} = 12\pi^2 n .
\]
Hence for sufficiently large $l_1$, the rectangle with side lengths $l_1$ and $l_2$ is not a maximizer.

Nonetheless, the rectangle can be a maximizer for some values of $l_1$ and $n$. For example, let $n=3$ and suppose the side lengths of the rectangle satisfy $l_2 \leq l_1 \leq \sqrt{8/3} \, l_2$. Then by simple comparisons we find
\begin{align*}
\lambda_1 & = \pi^2(l_1^{-2}+l_2^{-2}) , \\
\lambda_2 & = \pi^2(2^2l_1^{-2}+l_2^{-2}) , \\
\lambda_3 & = \pi^2(l_1^{-2}+2^2l_2^{-2}) ,
\end{align*}
and so
\[
(\lambda_1+\lambda_2+\lambda_3)\frac{A^3}{I} = 72\pi^2 .
\]
This value is the same as achieved for the square ($l_1=l_2$), and so there are many non-square maximizers when $n=3$.

The idea behind this construction is to identify a range of $(l_1,l_2)$ values for which the eigenvalues $\lambda_1,\lambda_2,\lambda_3$ have values $\pi^2(j^2l_1^{-2}+k^2l_2^{-2})$ for $(j,k)=(1,1),(2,1),(1,2)$. This set of index pairs in $\Z^2$ is invariant with respect to interchanging $j$ and $k$. Hence $\lambda_1+\lambda_2+\lambda_3$ is proportional to $l_1^{-2}+l_2^{-2}$, which allows us to cancel the denominator in $A^3/I$ and obtain an expression independent of the side lengths. This construction can of course be extended to arbitrarily large values of $n$, if desired.

\subsection*{Proof of Theorem~\ref{2dimR}}
The proof goes exactly as for the Dirichlet and Neumann cases in the proof of Theorem~\ref{2dimDN}, except that for the Robin eigenvalues we must take account also of a boundary integral in the Rayleigh quotient:
\begin{align*}
\frac{\int_{\partial E} v^2 \, ds(x)}{\int_E v^2 \, dx}
& = \frac{\int_{\partial E} u(UT^{-1}x)^2 \, ds(x)}{\int_E u(UT^{-1}x)^2 \, dx} \\
& = \frac{\int_{\partial D} u(Ux)^2 |T\tau(x)| \, ds(x)}{\int_D u(Ux)^2 \, dx \cdot |\det T|}
\end{align*}
by $x \mapsto Tx$, where $\tau(x)$ denotes the unit tangent vector to $\partial D$ at $x$. Geometrically, $|T\tau(x)|$ is the factor by which $T$ stretches the tangent direction to $\partial D$ at $x$.

The symmetry of $D$ implies that the tangent vectors rotate according to $\tau(U^{-1}x) = U^{-1} \tau(x)$, and so replacing $x$ with $U^{-1}x$ in the last integral gives
\[
\frac{\int_{\partial E} v^2 \, ds(x)}{\int_E v^2 \, dx}
=  |\det T|^{-1} \int_{\partial D} u(x)^2 |TU^{-1}\tau(x)| \, ds(x) .
\]
Choose $U$ to be the matrix $U_m$ representing rotation by angle $2\pi m/N$, for $m=1,\ldots,N$. Averaging the preceding quantity over $m$ and applying Cauchy--Schwarz gives the upper estimate
\begin{align}
& |\det T|^{-1} \int_{\partial D} u(x)^2 \Big\{ \frac{1}{N} \sum_{m=1}^N  |T U_m^{-1} \tau(x)|^2 \Big\}^{\! \! 1/2} \, ds(x) \label{cs} \\
& = |\det T|^{-1} \, \frac{1}{\sqrt{2}} \lVert T \rVert_{HS} \int_{\partial D} u(x)^2 \, ds(x) \notag
\end{align}
by Lemma~\ref{tightframematrix}, since $|\tau(x)|=1$. Multiplying by $\sigma \lVert T^{-1} \rVert_{HS}/\sqrt{2}$ gives
\[
\frac{1}{2} \lVert T^{-1} \rVert_{HS}^2 \, \sigma \int_{\partial D} u(x)^2 \, ds(x)
\]
by \eqref{inverseHS}.

With the aid of this last estimate we can straightforwardly adapt the proof of Theorem~\ref{2dimDN} to the Robin situation, and then call on Lemma~\ref{2dimHSnorm} to interpret the Hilbert--Schmidt norm of $T^{-1}$ in terms of moment of inertia.

\subsubsection*{Equality statement for Robin fundamental tone, $n=1$.} The proof of the equality statement follows the Dirichlet case in Theorem~\ref{2dimDN} up to the point where $N=4$ and
\[
u = \cos(\omega x_1) \cos(\omega x_2) ,
\]
and $\widetilde{D}$ contained in the open square ${\mathcal S} = (-\pi/2\omega,\pi/2\omega)^2$. We want to deduce a contradiction, so that the only way for equality to hold when $n=1$ is for $T$ to be a scalar multiple of an orthogonal matrix.

Equality must hold in the application of Cauchy--Schwarz at \eqref{cs}, except using $R$ instead of $T$ and using $\widetilde{D}$ instead of $D$. Hence
\begin{equation} \label{csequal}
|RU_1^{-1} \tau(x)|  = |RU_2^{-1} \tau(x)| = |RU_3^{-1} \tau(x)| = |RU_4^{-1} \tau(x)|
\end{equation}
for almost every (with respect to arclength measure) $x \in {\mathcal S} \cap \partial \widetilde{D}$; here we use that $u(x)^2>0$ on ${\mathcal S}$.

Consider such an $x$-value and write $\tau_1$ and $\tau_2$ for the components of the tangent vector $\tau(x)$. Then $|R \left( \begin{smallmatrix} \tau_1 \\ \tau_2 \end{smallmatrix} \right)|= |R \left( \begin{smallmatrix} -\tau_2 \\ \tau_1 \end{smallmatrix} \right)|$ by \eqref{csequal}, or
\[
(r_1 \tau_1)^2+(r_2 \tau_2)^2=(-r_1 \tau_2)^2+(r_2 \tau_1)^2 .
\]
Since $r_1^2 \neq r_2^2$, we can simplify to $\tau_1^2 = \tau_2^2$. Thus the tangent line at $x$ has slope $\pm 1$, and hence so does the normal vector.

The four possible normal vectors are $n(x)=(\e_1,\e_2)/\sqrt{2}$ where $\e_1,\e_2 \in \{ -1,1 \}$. Thus the Robin boundary condition $\frac{\partial u}{\partial n} + \sigma u = 0$ says
\[
\e_1 u_{x_1} + \e_2 u_{x_2} + \sqrt{2} \sigma u = 0 .
\]
Substituting $u = \cos(\omega x_1) \cos(\omega x_2)$ yields that
\[
\e_1 \tan(\omega x_1) + \e_2 \tan(\omega x_2) = \sqrt{2} \sigma/\omega .
\]
We conclude that every point $x \in {\mathcal S} \cap \partial \widetilde{D}$ lies on one of these four curves.

These curves have slope $\pm 1$ at only finitely many points in the square ${\mathcal S}$, and so
we conclude that no points of $\partial \widetilde{D}$ lie in that square. Hence $\partial \widetilde{D}$ lies entirely in the boundary of the square ${\mathcal S}$. The Robin condition fails on $\partial {\mathcal S}$, though, because $u=0$ there while $\frac{\partial u}{\partial n} \neq 0$.

This contradiction completes the proof.

\subsection*{Proof of Corollary~\ref{regularpolyR}}
First we show that among all triangles $G$, the quantity
\begin{equation} \label{robinbest}
(\rho_1 + \dots + \rho_n) \Big|_{\, \sigma \sqrt{\frac{I}{A^3}} \, , \, G} \left. \frac{A^3}{I} \right|_G
\end{equation}
is maximal for the equilateral triangle, for each $n \geq 1$. For this, let $D$ be an equilateral triangle centered at the origin, and let $G=T(D)$ where $T$ is an invertible linear transformation. Note that by Lemma~\ref{2dimHSnorm},
\[
\sigma \sqrt{\frac{I(D)}{A(D)^3}} \, \frac{\lVert T^{-1} \rVert_{HS}}{\sqrt{2}} = \sigma \sqrt{\frac{I(TD)}{A(TD)^3}} = \sigma \sqrt{\frac{I(G)}{A(G)^3}} .
\]
Hence replacing $\sigma$ in Theorem~\ref{2dimR} with $\sigma \sqrt{I(D)/A(D)^3}$ proves that expression \eqref{robinbest} is maximal when $T$ is the identity, that is, when $G$ is equilateral.

Next, assume the triangle $G$ has the same area as the equilateral $D$. Then the equilateral has smaller moment of inertia, $I(D) \leq I(G)$, as can be proved from formula \eqref{trianglemoment} for the moment of inertia along with Cauchy--Schwarz and the isoperimetric theorem for triangles. Hence by \eqref{robinbest} the expression
\[
(\rho_1 + \dots + \rho_n) \Big|_{\, \sigma \sqrt{\frac{I(D)}{A(D)^3}} \, , \, G} \left. \frac{A^3}{I} \right|_G
\]
is maximal when $G$ equals the equilateral triangle $D$. (Here we used monotonicity of the Robin eigenvalues with respect to the Robin parameter.) Finally, replacing $\sigma$ with $\sigma \sqrt{A(D)^3/I(D)}$ proves the corollary, for triangles.

Argue similarly for parallelograms and ellipses.

\subsection*{Proof of Theorem~\ref{2dimS}}
The proof proceeds as for the Dirichlet case in Theorem~\ref{2dimDN}, except that we must consider also the potential term in the numerator of the Rayleigh quotient. The key estimate is that
\begin{align*}
\frac{\int_E (W \circ T^{-1}) v^2 \, dx}{\int_E v^2 \, dx}
& = \frac{\int_D (W \circ U^{-1}) u^2 \, dx \cdot |\det TU^{-1}|}{\int_D u^2 \, dx \cdot |\det TU^{-1}|} \\
& = \int_D W u^2 \, dx
\end{align*}
by the rotational symmetry of $W$. The proof is now easily completed.

Incidentally, the assumption in the theorem that the potential $W$ should grow at infinity can be significantly weakened \cite{LSW09}.

\subsubsection*{Equality statement for Schr\"{o}dinger fundamental tone, $n=1$.}
Just like in the Dirichlet case, the singular value decomposition allows us to reduce to $T$ being diagonal. The analogues of equations \eqref{linear1} and \eqref{linear2} are that
\begin{align*}
h(u_{x_1 x_1} + u_{x_2 x_2}) & = \big( \widetilde{W}-{\mathcal E}_1(\widetilde{W}) \big) u , \\
\frac{2h}{r_1^{-2}+r_2^{-2}} \big( r_1^{-2} u_{x_1 x_1} + r_2^{-2} u_{x_2 x_2} \big) & = \big( \widetilde{W} - {\mathcal E}_1(\widetilde{W}) \big) u .
\end{align*}
(These equations hold pointwise a.e.\ by elliptic regularity theory, since the potential is locally bounded \cite[Theorem 8.8]{GT98}.) Solving these simultaneous equations, we deduce (since $r_1 \neq r_2$) that
\begin{equation} \label{eigensch}
u_{x_1 x_1} = u_{x_2 x_2} = \frac{1}{2h} \big( \widetilde{W} - {\mathcal E}_1(\widetilde{W}) \big) u .
\end{equation}
The potential $\widetilde{W}(x)$ is assumed to tend to $\infty$ as $|x| \to \infty$, and so $\widetilde{W}-{\mathcal E}_1>0$ whenever $|x|$ is sufficiently large. Multiplying \eqref{eigensch} by $u$ and integrating in the $x_1$ direction, we deduce that $-\int_\R u_{x_1}^2 \, dx_1 \geq 0$ when $|x_2|$ is sufficiently large, so that $u(x_1,x_2) = 0$ for almost every $x_1$. Since \eqref{eigensch} says that $u$ satisfies the one dimensional wave equation with $x_2$ playing the role of time variable and $x_1$ playing the role of space variable, we conclude that $u = 0$ a.e.\ in $\R^2$. This contradiction completes the proof.

\subsection*{Proof of Theorem~\ref{quadDN}}
We prove a generalization of Theorem~\ref{2dimDN}, namely that
\begin{equation} \label{DNeqquad}
(\lambda_1 + \dots + \lambda_n) \big( TD \big) \leq
\frac{1}{4} \big( \lVert T^{-1}_+ \rVert_{HS}^2 + \lVert T^{-1}_- \rVert_{HS}^2 \big) (\lambda_1 + \dots + \lambda_n) \big( D \big)
\end{equation}
for any bounded $D$ having rotational symmetry of order $N \geq 4$ with $N$ even.

The proof of Theorem~\ref{2dimDN} requires some modifications. First we show that pairwise orthogonality of the $v_j$ remains valid. Decomposing $D$ and $E=T(D)$ into their upper and lower halves $D_\pm = D \cap \R^2_\pm$ and $E_\pm = E \cap \R^2_\pm$, we compute
\[
\int_{E_\pm} v_j v_k \, dx = \int_{D_\pm} u_j u_k \, dx \cdot |\det T_\pm U^{-1}| .
\]
These upper and lower terms sum to zero because $\det T_+ = \det T_-$ and $\int_D u_j u_k \, dx = 0$ by assumption, when $j \neq k$. Thus $\int_E v_j v_k \, dx = 0$.

Next we consider the Rayleigh quotient of $v$. We decompose it as
\begin{equation} \label{decomposed}
\frac{\int_E |\nabla v|^2 \, dx}{\int_E v^2 \, dx}
= \sum_{\pm} \int_{U(D_\pm)} |(\nabla u)(x)UT^{-1}_\pm|^2 \, dx ,
\end{equation}
where in this calculation we use once more that the determinants of $T_+$ and $T_-$ agree.

Since $N$ is even, $U_{N/2}$ represents rotation by $\pi$, so that $U_{m+N/2}(D_\pm)=U_m(D_\mp)$ and $U_{m+N/2}=-U_m$. Hence when we average \eqref{decomposed} over the rotations $U=U_m$ we obtain
\begin{align*}
& \frac{1}{N} \sum_{m=1}^N \sum_\pm \int_{U_m(D_\pm)}  |(\nabla u)(x)U_mT^{-1}_\pm|^2 \, dx \\
& =  \sum_\pm \frac{1}{N} \sum_{m=1}^{N/2} \Big( \int_{U_m(D_\pm)} + \int_{U_m(D_\mp)}  \Big)  |(\nabla u)(x)U_mT^{-1}_\pm|^2 \, dx \\
& = \sum_\pm \int_D \frac{1}{N} \sum_{m=1}^{N/2}  |(\nabla u)(x)U_mT^{-1}_\pm|^2 \, dx \qquad \text{since $U_m(D)=D$} \\
& = \sum_\pm \int_D \frac{1}{2N} \sum_{m=1}^N |(\nabla u)(x)U_mT^{-1}_\pm|^2 \, dx \\
& = \sum_\pm \frac{1}{4} \lVert T^{-1}_\pm \rVert_{HS}^2 \int_D |\nabla u|^2 \, dx
\end{align*}
by Lemma~\ref{tightframematrix}. Now complete the proof of \eqref{DNeqquad} by recalling $u=u_j$ and summing over $j$.

Then the theorem follows from \eqref{DNeqquad} and the evaluation of the Hilbert--Schmidt norms in the next lemma.
\begin{lemma} \label{HSnormcombined}
Let $T$ be the piecewise linear homeomorphism in Theorem~\ref{quadDN}. If the bounded plane domain $D$ has rotational symmetry of order $N \geq 4$, with $N$ even, then
\[
\frac{1}{4} \left( \lVert T^{-1}_+ \rVert_{HS}^2 + \lVert T^{-1}_- \rVert_{HS}^2 \right) = \frac{I_0}{A^3}(TD) \Big/ \frac{I_0}{A^3}(D) .
\]
\end{lemma}
Recall $I_0$ denotes the moment of inertia about the origin.
\begin{proof}[Proof of Lemma~\ref{HSnormcombined}]
The moment integrals over the upper and lower halves of $D$ agree, with
\[
\int_{D_+} x_j x_k \, dx = \int_{D_-} x_j x_k \, dx , \qquad j,k =1,2 ,
\]
because $D_+$ maps to $D_-$ under rotation by $\pi$ (that is, $x \mapsto -x$). Here we use evenness of the order of rotation.

Hence the moment matrices satisfy $M(D_+)=M(D_-)=M(D)/2$. Since $M(D) = \frac{1}{2} I_0(D) \left( \begin{smallmatrix} 1 & 0 \\ 0 & 1 \end{smallmatrix} \right)$ , as shown in the proof of Lemma~\ref{2dimHSnorm}, we deduce
\begin{equation} \label{momentplusminus}
M(D_\pm) = \frac{1}{4} I_0(D) \begin{pmatrix} 1 & 0 \\ 0 & 1 \end{pmatrix}
\end{equation}
Now the moment of inertia of $TD$ about the origin can be computed as
\begin{align*}
I_0(TD)
& = \tr M(TD) \\
& = \tr M(T_+D_+) + \tr M(T_-D_-) \\
& = \sum_\pm \Big( \tr T_\pm M(D_\pm) T^\dagger_\pm \Big) \cdot |\det T_\pm| \\
& =\frac{1}{4} I_0(D) \sum_\pm \big( \tr T_\pm T^\dagger_\pm \big) \cdot |\det T_\pm|  \qquad \text{by \eqref{momentplusminus}} \\
& = \frac{1}{4} I_0(D) \left( \lVert T_+ \rVert_{HS}^2 + \lVert T_- \rVert_{HS}^2 \right) |\det T_\pm| ,
\end{align*}
where in the last step we used that $\det T_+ = \det T_-$.

The Hilbert--Schmidt norm of the inverse $T^{-1}_\pm$ is related to the Hilbert--Schmidt norm of $T_\pm$ by \eqref{inverseHS}, and so
\[
I_0(TD) = \frac{1}{4} I_0(D) \big( \lVert T^{-1}_+ \rVert_{HS}^2 + \lVert T^{-1}_- \rVert_{HS}^2 \big) |\det T_\pm|^3 ,
\]
from which the lemma follows.
\end{proof}

\section{\bf Literature on maximizing low eigenvalues under area, perimeter, in-radius or conformal mapping normalization} \label{literature}

This paper gives sharp upper bounds on the sum of the first $n \geq 1$ eigenvalues, normalized by $A^3/I$. To help put these results in context, we now describe results and conjectures that apply to the \emph{low} eigenvalues ($n=1,2,3$).

\subsection*{Dirichlet eigenvalues}
The quantity $\lambda_1A^2/L^2$ (where $L$ is the perimeter) is maximal among triangles for the equilateral triangle, by work of Siudeja \cite{S07}. This result is stronger than P\'{o}lya's upper bound \eqref{lambda1values} on $\lambda_1 A^3/I$, because $AL^2/I=36(l_1+l_2+l_3)^2/(l_1^2+l_2^2+l_3^3)$ by \eqref{trianglemoment} and this ratio is maximal for the equilateral triangle (when $l_1=l_2=l_3$).

Further, the normalized spectral gap $(\lambda_2-\lambda_1)A^2/L^2$ is maximal among triangles for the equilateral, by more recent work of Siudeja \cite{S10}, and thus $\lambda_2 A^2/L^2$ is maximal for the equilateral also. Hence $(\lambda_2-\lambda_1) A^3/I$ and $\lambda_2 A^3/I$ are maximal for the equilateral, which improves on Corollary~\ref{regularpolyDN} for $n=2$.

Among general convex domains, $\lambda_1 A^2/L^2$ is maximal for degenerate rectangles by work of P\'{o}lya \cite{P60}. That result differs from our Conjecture~\ref{conj1} on $\lambda_1 A^3/I$, where the equilateral triangle should also be a maximizer.
%An extension to doubly connected domains was observed by Osserman \cite{O77} (``the same
%argument can also be used for arbitrary simply-connected or doubly-connected domains (see
%Hersch \cite{H62}, p.~134)'').

Turning now to the indradius $R$, it is easy to see for triangles (or any polygon with an inscribed circle) that $A/L$ is proportional to $R$. Hence the preceding upper bounds for eigenvalues of triangles using $A^2/L^2$ can be restated using a normalizing factor of inradius squared. In particular, $\lambda_1 R^2$ is maximal for the equilateral triangle. A more general result is due to Solynin \cite{S93}: among all $N$-gons with an inscribed circle, $\lambda_1 R^2$ is maximal for the regular $N$-gon. Of course, among general domains the maximizer of $\lambda_1 R^2$ is simply the disk, by domain monotonicity.

For area normalization, Antunes and Freitas \cite[Conj.~6.1]{AF08} conjecture that the Faber--Krahn lower bound on $\lambda_1 A$ has a sharp upper analogue that includes an isoperimetric correction term: they conjecture that among simply connected plane domains,
\[
\lambda_1 A \leq \pi j_{0,1}^2 + \frac{\pi^2}{4} \Big( \frac{L^2}{A} - 4\pi \Big)
\]
with equality for the disk and (in a limiting sense) for degenerate rectangles.

Under a conformal mapping normalization, P\'{o}lya and Schiffer proved lower bounds on sums of \emph{reciprocal} eigenvalues $1/\lambda_1 + \cdots + 1/\lambda_n$, with the disk being extremal \cite{PS53}. Extensions to surfaces with bounded curvature were proved by Bandle \cite[p.\,120]{B80}, and to spectral zeta functions and doubly connected surfaces by Laugesen and Morpurgo \cite{L98,LM98}.

Lastly, the scale invariant ratio $\lambda_2/\lambda_1$ is maximal for the equilateral triangle among acute triangles, by work of Siudeja \cite{S10}. The conjecture remains open for obtuse triangles. For general domains, this Payne--P\'{o}lya--Weinberger functional is known to be maximal for the disk, by Ashbaugh and Benguria \cite{AB92}.

\subsection*{Neumann eigenvalues} Stronger inequalities are known than the one we found for $\mu_2 A^3/I$ in \eqref{neumanntone} (which is the case $n=2$ of Corollary~\ref{regularpolyDN}). In fact, $\mu_2 A$ is maximal for the equilateral triangle among triangles, and for the square among parallelograms, and for the disk among all bounded plane domains. The first of these stronger inequalities was proved recently by the authors \cite[Theorem~3.1]{LS09}. The second, for parallelograms, is unpublished work of the authors. The third inequality is a result of Szeg\H{o} and Weinberger \cite{S54,W56}. These inequalities for $\mu_2 A$ are stronger because $A^2/I$ is maximal for the equilateral triangle among triangles, for the square among parallelograms, and for the disk among all domains.

Our inequalities in Corollary~\ref{regularpolyDN} hold for all $n \geq 2$. In contrast, the stronger inequalities fail to extend to $n=3$. The maximizing domains are instead somewhat elongated: the ``arithmetic mean'' $(\mu_2+\mu_3)A$ seems to be maximal among isosceles triangles for an aperture slightly greater than $\pi/6$ (according to numerical work), rather than for the equilateral triangle with aperture $\pi/3$; and $(\mu_2+\mu_3)A$ seems to be maximal among parallelograms for the 2:1 rectangle rather than the square (see \cite[\S5]{AB93} for comments on rectangles). The maximizer among convex domains is apparently not known. The only positive result is that the disk is maximal among $4$-fold symmetric domains \cite[\S4]{AB93}. Incidentally, it is open to maximize the geometric mean $\sqrt{\mu_2 \mu_3} A$. The disk is conjectured to be extremal, by I. Polterovich.

Among convex plane domains, it is open to maximize $\mu_2 L^2$. The disk is not the maximizer, because the equilateral triangle and the square have a larger value (in fact, the same value). The maximizer for $\mu_2 D^2$, where $D$ is diameter, is known to be the degenerate obtuse isosceles triangle by work of Cheng \cite[Theorem~2.1]{C75}, \cite[Proposition 3.6]{LS10}. For the problems mentioned above, and for related conjectures on triangles, see \cite[\S{IX}]{LS09}.

Sums of reciprocal Neumann eigenvalues were minimized by Dittmar \cite{D09}, under conformal mapping normalization.

\subsection*{Lower bounds} Sharp \emph{lower} bounds on Dirichlet eigenvalue sums for triangles are proved in a companion paper \cite{LS11a}, under diameter normalization. Lower bounds for the Neumann eigenvalue $\mu_2$ are found in an earlier work \cite{LS10}. References to other lower bounds can be found in those papers.

\section*{Acknowledgments} We are grateful to Mark Ashbaugh for guiding us to relevant literature.

\appendix

\section{\bf Eigenvalues of equilateral triangles, rectangles, disks} \label{eigenapp}

The Dirichlet eigenfunctions of equilateral triangles were derived
about 150 years ago by Lam\'{e} \cite[pp.~131--135]{L66}. (See the treatment in the text of Mathews and Walker \cite[pp.~237--239]{MW70} or in the paper by Pinsky \cite{Pi85}. Note also the recent exposition by McCartin \cite{M03}.) Dirichlet eigenfunctions of rectangles and disks are well known too \cite{B80}. The eigenvalues are:
\begin{align*}
\big\{ (16\pi^2/9) \big[ j_1^2 + j_1j_2 + j_2^2 \big] : j_1,j_2 \geq 1 \big\} & \quad \text{for an equilateral triangle of side $1$,} \\
\big\{ \pi^2 \big[ (j_1/l_1)^2 + (j_2/l_2)^2 \big] : j_1,j_2 \geq 1 \big\} & \quad \text{for a rectangle of side lengths $l_1,l_2$,} \\
\big\{ j_{m,p}^2 : m \geq 0, p \geq 1 \big \} & \quad \text{for the unit disk,}
\end{align*}
where $j_{m,p}$ is the $p$th zero of the Bessel function $J_m$.

The Neumann eigenvalues are:
\begin{align*}
\big\{ (16\pi^2/9) \big[ j_1^2 + j_1j_2 + j_2^2 \big] : j_1,j_2 \geq 0 \big\} & \quad \text{for an equilateral triangle of side $1$,} \\
\big\{ \pi^2 \big[ (j_1/l_1)^2 + (j_2/l_2)^2 \big] : j_1,j_2 \geq 0 \big\} & \quad \text{for a rectangle of side lengths $l_1,l_2$,} \\
\big\{ (j^\prime_{m,p})^2 : m \geq 0, p \geq 1 \big \} & \quad \text{for the unit disk,}
\end{align*}
where $j^\prime_{m,p}$ is the $p$th zero of the Bessel derivative $J_m^\prime$. See \cite{B80,M02}.

The Robin eigenvalues of rectangles and disks can be found by separation of variables. The eigenvalues are known also for the equilateral triangle \cite{M04}.

\end{document}